\documentclass{amsart}
\usepackage{latexsym}
\usepackage{amssymb,amsmath,amsfonts}
\newtheorem{theorem}{Theorem}[section]
  \newtheorem{lemma}[theorem]{Lemma}

  \theoremstyle{definition}
  \newtheorem{definition}[theorem]{Definition}

  \theoremstyle{remark}
  \newtheorem{remark}[theorem]{Remark}

  \numberwithin{equation}{section}
  \begin{document}
  \thispagestyle{empty}
  \title[Convergence Theorems]
   {Convergence results for a common solution of a finite family of equilibrium problems and quasi-Bregman nonexpansive mappings in Banach space}
  \date{}
    \author[G.C. Ugwunnadi and Bashir Ali]
  { G.C. Ugwunnadi$^{1}$ and  Bashir Ali$^{2}$  }
\address{$~^{1}$Department of Mathematics, \newline \indent \hspace{2mm} Michael Okpara University of Agriculture, \newline \indent \hspace{2mm} Umudike, Abia State, Nigeria} \email{ugwunnadi4u@yahoo.com}
\address{$~^{2}$Department of Mathematical Sciences, \newline \indent \hspace{2mm} Bayero University Kano\newline \indent \hspace{2mm} P.M.B. 3011  Kano, Nigeria} \email{bashiralik@yahoo.com}

\begin{abstract}
In this paper, we introduce an iterative process for finding common fixed point of finite
family of quasi-Bregman nonexpansive mappings which is a unique solution of some equilibrium problem .

 \end{abstract}
\maketitle
\section{Introduction}
\noindent Let $E$ be a real reflexive Banach space, $C$ a nonempty subset of $E$. Let $T:C\to C$ be a map, a point $x\in C$ is called a fixed point of $T$ if $Tx=x$, and the set of all fixed points of $T$ is denoted by $F(T)$. The mapping $T$ is called $L-$Lipschitzian or simply Lipschitz if there exists $L>0$, such that $||Tx-Ty||\leq L||x-y||,~\forall x,y\in C$ and if $L=1$, then the map $T$ is called nonexpansive.\\[2mm]
Let $g:C\times C\to\mathbb{R}$ be a bifunction. The equilibrium problem with respect to $g$ is to find $$z\in C\quad\textrm{such that}\quad g(z,y)\geq0,~\forall y\in C.$$
The set of solution of equilibrium problem is denoted by $EP(g).$ Thus
$$EP(g):=\{z\in C: g(z,y)\geq0,~\forall y\in C\}.$$
Numerous problems in Physics, Optimization and Economics reduce to finding a solution of the equilibrium problem. Some methods have been proposed to solve equilibrium problem in Hilbert spaces; see for example Blum and Oettli \cite{blum}, Combettes and Hirstoaga \cite{com}. Recently, Tada and Takahashi \cite{tada,tak1} and Takahashi and Takahashi \cite{stak} obtain weak and strong convergence theorems for finding a common element of the set of solutions of an equilibrium problem and set of fixed points of nonexpansive mapping in Hilbert space. In particular, Tada and Takahashi \cite{tak1} establish a strong convergence theorem for finding a common element of the two sets by using the hybrid method introduced in Nakajo and Takahashi \cite{nakajo}. They also proved such a strong convergence theorem in a uniformly convex and uniformly smooth Banach space.\\[2mm]
In 1967, Bregman \cite{bregman} discovered an elegant and effective technique for using so-called Bregman distance
function $D_f$ see, \eqref{eq1} in the process of designing and analyzing feasibility and optimization
algorithms. This opened a growing area of research in which Bregman's technique has been applied in various ways in order to design and analyze iterative algorithms for solving feasibility and optimization
problems.\\
Let $f:E\to (-\infty,+\infty]$ be a convex and G$\hat{a}$teaux differentiable function. The function $D_{f}:\textrm{dom} f\times \textrm{int dom} f\to[0,+\infty)$ defined as follows:
\begin{eqnarray}\label{eq1}
D_{f}(y,x):= f(y)-f(x)-\langle\nabla f(x),y-x\rangle
\end{eqnarray}
is called the Bregman distance with respect to $f$ (see \cite{cen}). It is obvious from the definition of $D_{f}$ that
\begin{eqnarray}\label{eq2}
D_{f}(z,x)=D_{f}(z,y)+D_{f}(y,x)+\langle\nabla f(y)-\nabla f(x),z-y\rangle.
\end{eqnarray}
We observed from \eqref{eq2}, that for any $y_{1},y_{2},\cdots,y_{N}\in E$, the following holds
\begin{eqnarray}\label{projsum}
D_{f}(y_{1},y_{N})=\sum^{N}_{k=2}D_{f}(y_{k-1},y_{k})+\sum^{N}_{k=3}\langle\nabla f(y_{k-1})-\nabla f(y_{k}),y_{k-1}-y_{1}\rangle.
\end{eqnarray}

Recall that the Bregman projection \cite{bregman} of $x\in\textrm{int dom} f$ onto the nonempty closed and convex set $C\subset \textrm{dom} f$ is the necessarily unique vector $P^{f}_{C}(x)\in C$ satisfying
$$D_{f}(P^{f}_{C}(x),x)=\inf\{D_{f}(y,x):y\in C\}.$$
A mapping $T$ is said to be Bregman firmly nonexpansive \cite{reich3}, if for all $x,y\in C,$
$$\langle\nabla f(Tx)-\nabla f(Ty),Tx-Ty\rangle\leq\langle\nabla f(x)-\nabla f(y),Tx-Ty\rangle$$
or equivalently,
$$D_{f}(Tx,Ty)+D_{f}(Ty,Tx)+D_{f}(Tx,x)+D_{f}(Ty,y)\leq D_{f}(Tx,y)+D_{f}(Ty,x).$$
A point $p\in C$ is said to be asymptotic fixed point of a map $T$, if for any sequence $\{x_n\}$ in $C$ which converges weakly to $p$, and $\underset{n\to\infty}{\lim}||x_n-Tx_n||=0$. We denote by $\hat{F}(T)$ the set of asymptotic fixed points of $T$. Let $f:E\to\mathbb{R}$, a mapping $T:C\to C$ is said to be Bregman relatively nonexpansive \cite{mat} if $F(T)\neq\empty,\hat{F}(T)=F(T)$ and $D_{f}(p,T(x))\leq D_{f}(p,x)$ for all $x\in C$ and $p\in F(T)$. $T$ is said to be quasi-Bregman relatively nonexpansive if $F(T)\neq\emptyset,$ and $D_{f}(p,T(x))\leq D_{f}(p,x)$ for all $x\in C$ and $p\in F(T)$.\\[2mm]
Recently, by using the Bregman projection, in 2011 Reich and Sabach \cite{reich3} proposed  algorithms for
finding common fixed points of finitely many Bregman firmly nonexpansive operators in a reflexive Banach space.
\begin{eqnarray}\label{1.4}
 {} \left\{ \begin{array}{ll}
 x_0\in E  & \textrm{ $  $}
 \\ Q^{i}_{0}=E, i=1,2,\cdots,N & \textrm{ $  $}
 \\ u_{n}\in C ~\textrm{such that} & \textrm{ $  $}
 \\ y^{i}_{n}=T_{i}(x_{n}+e^{i}_{n}),  & \textrm{ $  $}
 \\ Q^{i}_{n+1}=\{z\in Q^{i}_{n}:\langle\nabla f(x_n+e^{i}_{n})-\nabla f(y^{i}_{n}),z-y^{i}_{n}\rangle\leq 0\}, & \textrm{ $  $}
 \\ C_{n}=\bigcap^{N}_{i=1}C_{n}^{i},  & \textrm{ $  $}
 \\ x_{n+1}=P^{f}_{C_{n+1}}x_{0},n\geq0.
\end{array}  \right.
\end{eqnarray}
Under some suitable conditions, they proved that the sequence  generated by \eqref{1.4} converges strongly to $\bigcap^{N}_{i=1}F(T_i)$ and applied the result for the solution of convex feasibility and equilibrium problems.\\[2mm]
 In 2011, Chen et al. \cite{20}, introduced the concept of weak Bregman relatively nonexpansive mappings in a reflexive Banach space and gave an example to illustrate the existence of a weak Bregman relatively nonexpansive
mapping and the difference between a weak Bregman relatively nonexpansive mapping and a Bregman relatively
nonexpansive mapping. They also proved strong convergence of the sequences generated by the constructed
algorithms with errors for finding a fixed point of weak Bregman relatively nonexpansive mappings and Bregman
relatively nonexpansive mappings under some suitable conditions.\\
\noindent Recently in 2014, Alghamdi et al. \cite{Alg} proved a strong convergence theorem for the common fixed point of finite family of quasi-Bregman nonexpansive mappings. Pang $et\:al.$ \cite{Panc} proved weak convergence theorems for Bregman relatively nonexpansive mappings. While, Zegeye and Shahzad in \cite{Zeg} and \cite{Zege} proved a strong convergence theorem for the common fixed point of finite family of right Bregman strongly nonexpansive mappings and Bregman weak relatively nonexpansive mappings in reflexive Banach space respectively.\\

In 2015 Kumam et al.\cite{kumam} introduced the following algorithm:
\begin{eqnarray}\label{kum}
 {} \left\{ \begin{array}{ll}
 x_1=x\in C  & \textrm{ $  $}
 \\ z_{n}=Res^{f}_{g}(x_{n}) & \textrm{ $  $}
 \\ y_{n}=\nabla f^{*}(\beta_{n}\nabla f^{*}(x_n)+(1-\beta_{n})\nabla f^{*}(T_{n}(z_n))) & \textrm{ $  $}
 \\ x_{n+1}=\nabla f^{*}(\alpha_{n}\nabla f^{*}(x_n)+(1-\alpha_{n})\nabla f^{*}(T_{n}(y_n))),
\end{array}  \right.
\end{eqnarray}
where $T_{n},n\in\mathbb{N}$, is a Bregman strongly nonexpansive mapping. They proved that the sequence $\{x_{n}\}$ which is generated by the algorithm \eqref{kum} converges strongly to the point $P_{\Omega}^{f}x$, where $\Omega:=F(T)\cap EP(g).$

Motivated and inspired by the above works, in this paper, we prove a new strong convergence theorem for
finite family of quasi-Bregman nonexpansive mapping and system of equilibrium problem in a real Banach space.

\section{Preliminaries}
\noindent Let $E$ be a real reflexive Banach space with the norm $||.||$ and $E^*$ the dual space of $E$. Throughout this paper, we shall assume $f:E\to(-\infty,+\infty]$ is a proper, lower semi-continuous and convex function. We denote by dom$f~:=\{x\in E:f(x)<+\infty\}$ as the domain of $f$.\\
Let $x\in\textrm{int dom}f$, the subdifferential of $f$ at $x$ is the convex set defined by
$$\partial f(x)=\{x^*\in E^*:f(x)+\langle x^*,y-x\rangle\leq f(y),~\forall y\in E\},$$
where the Fenchel conjugate of $f$ is the function $f^*:E^*\to(-\infty,+\infty]$ defined by
$$f^*(x^*)=\sup\{\langle x^*,x\rangle-f(x): x\in E\}.$$
We know that the Young-Fenchel inequality holds:
$$\langle x^*,x\rangle\leq f(x)+f^*(x^*),~~\forall x\in E,~x^*\in E^*.$$
A function $f$ on $E$ is coercive \cite{hir} if the sublevel set of $f$ is bounded; equivalently,
$$\underset{||x||\to+\infty}{\lim}f(x)=+\infty.$$
A function $f$ on $E$ is said be strongly coercive \cite{zal} if
$$\underset{||x||\to+\infty}{\lim}\frac{f(x)}{||x||}=+\infty.$$
For any $x\in \textrm{int dom}f$ and $y\in E$, the right-hand derivative of $f$ at $x$ in the direction $y$ is defined by
$$f^{\circ}(x,y):=\underset{t\to0^+}{\lim}\frac{f(x+ty)-f(x)}{t}.$$
The function $f$ is said to be G$\hat{a}$teaux differentiable at $x$ if ${\lim}_{t\to0^+}\frac{f(x+ty)-f(x)}{t}$ exists for any $y$. In this case, $f^{\circ}(x,y)$ coincides with $\nabla f(x)$, the value of the gradient $\nabla f$ of $f$ at $x$. The function $f$ is said to be G$\hat{a}$teaux differentiable if it is G$\hat{a}$eaux differentiable for any $x\in\textrm{int dom}f$. The function $f$ is said to be Fr$\acute{e}$chet differentiable at $x$ if this limit is attained uniformly in $||y||=1.$ Finally, $f$ is said to be uniformly Fr$\acute{e}$chet differentiable on a subset $C$ of $E$ if the limit is attained uniformly for $x\in C$ and $||y||=1$. It is known that if $f$ is G$\hat{a}$teaux differentiable (resp. Fr$\acute{e}$chet differentiable) on int dom$f$, then $f$ is continuous and its G$\hat{a}$teaux derivative $\nabla f$ is norm-to-weak$^*$ continuous (resp. continuous) on int dom$f$ (see also \cite{asp,bon}). We will need the following results.
\begin{lemma}\cite{reich}\label{le1}
If $f:E\to\mathbb{R}$ is uniformly Fr$\acute{e}$chet differentiable and bounded on bounded subsets of $E$, then $\nabla f$ is uniformly continuous on bounded subsets of $E$ from the strong topology of $E$ to the strong topology of $E^*$.
\end{lemma}
\begin{definition}\cite{bau}\label{2}
The function $f$ is said to be:
\begin{itemize}
\item [(i)] essentially smooth, if $\partial f$ is both locally bounded and single-valued on its domain.
\item [(ii)] essentially strictly convex, if $(\partial f)^{-1}$ is locally bounded on its domain and $f$ is strictly convex on every convex subset of dom$\partial f$.
    \item [(iii)] Legendre, if it is both essentially smooth and essentially strictly convex.
    \end{itemize}
    \end{definition}
\begin{remark} Let $E$ be a reflexive Banach space. Then we have
\begin{itemize}
\item [(i)] $f$ is essentially smooth if and only if $f^{*}$ is essentially strictly convex (see \cite{bau}, Theorem 5.4).
\item [(ii)] $(\partial f)^{-1}=\partial f^{*}$ (see \cite{bon})
\item [(iii)] $f$ is Legendre if and only if $f^{*}$ is Legendre, (see \cite{bau}, Corollary 5.5).
\item [(iv)] If $f$ is Legendre, then $\nabla f$ is a bijection satisfying\\ $\nabla f~=(\nabla f^*)^{-1}$, ran $\nabla f~=$ dom $\nabla f^*~=$ int dom $f^*$ and ran $\nabla f^*~=$ dom $  f~=$ int dom $f$, (see \cite{bau}, Theorem 5.10).
    \end{itemize}
    \end{remark}

    The following result was prove in \cite{narag}, (see also \cite{narag1}).
  \begin{lemma}\label{jen} Let $E$ be a Banach space, $r>0$ be a constant, $\rho_{r}$ be the gauge of uniform convexity of $g$ and $g:E\to\mathbb{R}$ be a convex function which is uniformly convex on bounded subsets of $E$. Then
\begin{itemize}
\item[(i)] For any $x,y\in B_{r}$ and $\alpha\in(0,1)$,
$$g(\alpha x+(1-\alpha)y)\leq\alpha g(x)+(1-\alpha)g(y)-\alpha(1-\alpha)\rho_{r}(||x-y||).$$
\item[(ii)] For any $x,y\in B_{r}$,
$$\rho_{r}(||x-y||)\leq D_{g}(x,y)$$
\item[(iii)] If, in addition, $g$ is bounded on bounded subsets and uniformly convex on bounded subsets of $E$ then, for any $x\in E, y^*,z^*\in B_{r}$ and $\alpha\in(0,1)$,
  $$V_{g}(x,\alpha y^*+(1-\alpha)z^*)\leq\alpha V_{g}(x,y^*)+(1-\alpha)V_{g}(x,z^*)-\alpha(1-\alpha)\rho^*_{r}(||y^*-x^*||).$$
  \end{itemize}
  \end{lemma}

\begin{lemma}(\cite{jenchi})
Let $E$ be a Banach space, let $r>0$ be a constant and let $f:E\to \mathbb{R}$ be a continuous and convex function which is uniformly convex on bounded subsets of $E$. Then
$$f\Big(\sum^{\infty}_{k=0}\alpha_{k}x_{k}\Big)\leq\sum^{\infty}_{k=0}\alpha_{k}f(x_{k})-\alpha_{i}\alpha_{j}\rho_{r}(||x_{i}-x_{j}||)$$
for all $i,j\in \mathbb{N}\cup\{0\},x_{k}\in B_{r},\alpha_{k}\in(0,1)$ and $k\in\mathbb{N}\cup\{0\}$ with $\sum^{\infty}_{k=0}\alpha_{k}=1$, where $\rho_{r}$ is the gauge of uniform convexity of $f.$
\end{lemma}
We know the following two results; see \cite{zal}
\begin{theorem}\label{jen1}
Let $E$ be a reflexive Banach space and let $f: E \to\mathbb{R}$ be a convex function
which is bounded on bounded subsets of $E$. Then the following assertions are equivalent:
\begin{itemize}
\item[(1)] $f$ is strongly coercive and uniformly convex on bounded subsets of $E$;
\item[(2)] $dom f^* = E^*$, $f^*$ is bounded on bounded subsets and uniformly smooth on bounded
subsets of $E^*$;
\item[(3)] $dom f^* = E^*, f^*$ is Frechet differentiable and $\nabla f$ is uniformly norm-to-norm
continuous on bounded subsets of $E^*$.\end{itemize}
\end{theorem}
\begin{theorem}\label{jen2}
 Let $E$ be a reflexive Banach space and let $f : E\to\mathbb{R}$ be a continuous convex
function which is strongly coercive. Then the following assertions are equivalent:
\begin{itemize}\item[(1)] $f$ is bounded on bounded subsets and uniformly smooth on bounded subsets of $E$;
\item[(2)] $f^*$ is Frechet differentiable and $ f^*$ is uniformly norm-to-norm continuous on
bounded subsets of $E^*$;
\item[(3)] $dom f^* = E^*, f^*$ is strongly coercive and uniformly convex on bounded subsets of $E^*$.\end{itemize}
\end{theorem}
The following result was first proved in \cite{but} (see also \cite{koh}).
\begin{lemma}\label{jen3} Let $E$ be a reflexive Banach space, let $f : E\to\mathbb{R}$ be a strongly coercive Bregman
function and let $V$ be the function defined by
$$V(x,x^*)=f(x)-\langle x,x^*\rangle+f^*(x^*),~~x\in E,~~x^*\in E^*$$
Then the following assertions hold:
\begin{itemize}\item[(1)] $D_{f} (x,\nabla f(x^*)) = V(x, x^*)$ for all $x\in E$ and $x^*\in E^*$.
\item[(2)] $V(x, x^*) +\langle\nabla f^*(x^*)-x,y^*\rangle\leq V(x, x^* + y^*)$ for all $x \in E$ and $x^*, y^*\in E^*$.
\end{itemize}
\end{lemma}

Examples of Legendre functions were given in \cite{bau,bau1}. One important and interesting Legendre function is $\frac{1}{p}||\cdot||^{p}(1< p<\infty)$ when $E$ is a smooth and strictly convex Banach space. In this case the gradient $\nabla f$ of $f$ is coincident with the generalized duality mapping of $E$, i.e., $ \nabla f = J_{p}(1< p< \infty)$. In particular, $\nabla f= I$ the identity mapping in Hilbert spaces. In the rest of this paper, we always assume that $f:E\to (-\infty,+\infty]$ is Legendre.\\[2mm]
Concerning the Bregman projection, the following are well known.
\begin{lemma}\cite{but}\label{le2} Let $C$ be a nonempty, closed and convex subset of a reflexive Banach space $E$. Let $f:E\to\mathbb{R}$ be a G$\hat{a}$teaux differentiable and totally convex function and let $x\in E.$ Then
\begin{itemize}
\item [(a)] $z=P^{f}_{C}(x)$ if and only if $\langle \nabla f(x)-\nabla f(z),y-z\rangle\leq0,~\forall y\in C.$
\item [(b)] $D_{f}(y,P^{f}_{C}(x))+D_{f}(P^{f}_{C}(x),x)\leq D_{f}(y,x),~\forall x\in E,~y\in C.$
\end{itemize}
\end{lemma}
Let $f:E\to(-\infty,+\infty]$ be a convex and G$\hat{a}$teaux differentiable function. The modulus of total convexity of $f$ at $x\in \textrm{int dom} f$ is the function $v_{f}(x,\cdot):[0,+\infty)\to[0,+\infty]$ define by
$$v_{f}(x,t):=\inf\{D_{f}(y,x):y\in\textrm{dom} f,||y-x||=t\}.$$
The function $f$ is called totally convex at $x$ if $v_{f}(x,t)>0$ whenever $t>0$. The function $f$ is called totally convex if it is totally convex at any point $x\in \textrm{int dom} f$ and is said to be totally convex on bounded sets if $v_{f}(B,t)>0$ for any nonempty bounded subset $B$ of $E$ and $t>0$, where the modulus of total convexity of the function $f$ on the set $B$ is the function $v_{f}:\textrm{int dom} f\times [0,+\infty)\to[0,+\infty]$ defined by $$v_{f}(B,t):=\inf\{v_{f}(x,t):x\in B\cap\textrm{dom} f\}.$$
\begin{lemma}\cite{res}\label{le3} If $x\in\textrm{dom} f$, then the following statements are equivalent:
\begin{itemize}
\item [(i)] The function $f$ is totally convex at $x$;
\item [(ii)] For any sequence $\{y_n\}\subset\textrm{dom} f$, $$\underset{n\to+\infty}{\lim}D_{f}(y_n,x)=0\Rightarrow\underset{n\to+\infty}{\lim}||y_n-x||=0.$$
    \end{itemize}\end{lemma}
    Recall that the function $f$ called sequentially consistent \cite{but} if for any two sequence $\{x_n\}$ and $\{y_n\}$ in $E$ such that the first one is bounded
     $$\underset{n\to+\infty}{\lim}D_{f}(y_n,x_n)=0\Rightarrow\underset{n\to+\infty}{\lim}||y_n-x_n||=0.$$
\begin{lemma}\cite{but1}\label{le4} The function $f$ is totally convex on bounded sets if and only if the function $f$ is sequentially consistent.
\end{lemma}
\begin{lemma}\cite{reich1} \label{le5}Let $f:E\to\mathbb{R}$ be a G$\hat{a}$teaux differentiable and totally convex function. If $x_0\in E$ and the sequence $\{D_{f}(x_n,x_0)\}$ is bounded, then the sequence $\{x_{n}\}$ is bounded too.
\end{lemma}
\begin{lemma}\cite{reich1}\label{le6} Let $f:E\to\mathbb{R}$ be a G$\hat{a}$teaux differentiable and totally convex function, $x_0\in E$ and let $C$ be a nonempty, closed and convex subset of $E$. Suppose that the sequence $\{x_n\}$ is bounded and any weak subsequential limit of $\{x_n\}$ belongs to $C$. If $D_{f}(x_n,x_0)\leq D_{f}(P^{f}_{C}(x_0),x_0)$ for any $n\in\mathbb{R}$, then $\{x_n\}$ converges strongly to $P^{f}_{C}(x_0)$.\end{lemma}
\begin{lemma}\cite{phe}\label{le7} Let $E$ be a real reflexive Banach space, $f:E\to(-\infty,+\infty]$ be a proper lower semi-continuous function, then $f^*:E^*\to(-\infty,+\infty]$ is a proper weak$^*$ lower semi-continuous and convex function. Thus, for all $z\in E,$ we have
\begin{eqnarray}\label{d}
D_{f}(z,\nabla f^*(\sum^{N}_{i=1}t_i\nabla f(x_i)))\leq\sum^{N}_{i=1}t_iD_{f}(z,x_i)
\end{eqnarray}\end{lemma}
In order to solve the equilibrium problem, let us assume that a bifunction $g:C\times C\to\mathbb{R}$ satisfies the following condition \cite{blum}
\begin{itemize}
\item [(A1)] $g(x,x)=0,~\forall x\in C.$
\item [(A2)] $g$ is monotone, i.e., $g(x,y)+g(y,x)\leq0,~\forall x,y\in C.$
\item [(A3)] $\limsup_{t\downarrow0} g(x+t(z-x),y)\leq g(x,y)~\forall x,z,y\in C.$
\item [(A4)] The function $y\mapsto g(x,y)$ is convex and lower semi-continuous.
\end{itemize}
The resolvent of a bifunction $g$ \cite{com} is the operator $Res^{f}_{g}:E\to2^{C}$ defined by
\begin{eqnarray}\label{bi}Res^{f}_{g}(x)=\{z\in C:g(z,y)+\langle\nabla f(z)-\nabla f(x),y-z\rangle\geq0,~\forall y\in C\}.\end{eqnarray}
From (Lemma 1, in \cite{reich2}), if $f:(-\infty,+\infty]$ is a strongly coercive and G$\hat{a}$teaux differentiable function, and $g$ satisfies conditions (A1)-(A4), then dom$(Res^{f}_{g})=E$. The following lemma gives some characterization of the resolvent $Res^{f}_{g}$.
\begin{lemma}\cite{reich2}\label{le8} Let $E$ be a real reflexive Banach space and $C$ be a nonempty closed convex subset of $E$. Let $f:E\to(-\infty,+\infty]$ be a Legendre function. If the bifunction $g:C\times C\to\mathbb{R}$ satisfies the conditions (A1)-(A4). Then, the followings hold:
\begin{itemize}
\item[(i)] $Res^{f}_{g}$ is single-valued;
\item [(ii)] $Res^{f}_{g}$ is a Bregman firmly nonexpansive operator;
\item[(iii)] $F(Res^{f}_{g})=EP(g)$;
\item [(iv)] $EP(g)$ is closed and convex subset of $C$;
\item [(v)] for all $x\in E$ and for all $q\in F(Res^{f}_{g})$, we have
\begin{eqnarray}\label{11}
D_{f}(q,Res^{f}_{g}(x))+D_{f}(Res^{f}_{g}(x),x)\leq D_{f}(q,x).
\end{eqnarray}
\end{itemize}
\end{lemma}
\begin{lemma}(\cite{xu6})\label{xu6} Let $\{a_n\}$ be a sequence of nonnegative real numbers satisfying the following relation:
$$a_{n+1}\leq(1-\alpha_{n})a_{n}+\alpha_{n}\delta_{n},~~n\geq n_0,$$
where $\{\alpha_{n}\}\subset (0,1)$ and $\{\delta_{n}\}$ is a real sequence satisfying the following conditions:
$$\underset{n\to\infty}{\lim}\alpha_{n}=0,\sum^{\infty}_{n=1}=\infty,~~\textrm{as}~~\underset{n\to\infty}{\limsup}~\delta_{n}\leq0.$$
Then, $\underset{n\to\infty}{\lim}a_{n}=0$.
\end{lemma}
\begin{lemma}(\cite{mainge})\label{mainge} Let $\{a_{n}\}$ be a sequence of real numbers such that there exists a subsequence $\{n_i\}$ of $\{n\}$ such that $a_{n_i}< a_{n_i+1}$ for all $i\in\mathbb{N}$. Then there exists a nondecreasing sequence $\{m_k\}\subset \mathbb{N}$ such that $m_{k}\to\infty$ and the following properties are satisfied by all (sufficiently large) numbers $k\in\mathbb{N}$.
$$a_{m_k}\leq a_{m_k+1}~~\textrm{and}~~a_k\leq a_{m_k+1}.$$
In fact, $m_{k}=\max\{j\leq k:a_j<a_{j+1}\}$.
\end{lemma}

\section{Main Results}

We now prove the following theorem.
\begin{theorem} Let $C$ be a nonempty, closed and convex subset of a real reflexive Banach space $E$ and $f:E\to\mathbb{R}$ a strongly coercive Legendre function which is bounded, uniformly Fr$\acute{e}$chet differentiable and totally convex on bounded subset of $E$. For each $j=1,2,\cdots,m$, let $g_{j}$ be a bifunction from $C\times C$ to $\mathbb{R}$ satisfying (A1)-(A4) and let $\{T_{i=1}^{N}\}$ be a finite family of quasi-Bregman nonexpansive self mapping of $C$ such that
$F:=\cap^{N}_{i=1}F(T_{i})\neq\emptyset,$ where $F=F(T_{N}T_{N-1}T_{N-2}\cdots T_{2}T_{1})=F(T_{1}T_{N}T_{N-1}T_{N-2}\cdots T_{2})=\cdots =F(T_{N-1}T_{N-2}\cdots T_{2}T_{1}T_{N})\neq\emptyset$ and\\
$\Omega:=\Big(\cap^{m}_{j=1}EP(g_j)\Big)\bigcap F\neq\emptyset.$ Let $\{x_{n}\}^{\infty}_{n=1}$ be a sequence generated by $x_1 =x\in C,C_{1}=C$ and
\begin{eqnarray}\label{rec}
 {} \left\{ \begin{array}{ll}
 x_{1}\in C & \textrm{ $  $}
  \\ u_{j,n}=Res^{f}_{g_j}x_n,~~j=1,2,3,\cdots,m  & \textrm{ $  $}
 \\ y_{n}=P_{C}(\nabla f^*((1-\alpha_{n})\nabla f(u_{j,n})))  & \textrm{ $  $}
 \\ x_{n+1}=P_{C}(\nabla f^*(\beta_n\nabla f(y_{n})+(1-\beta_n)\nabla f(T_{[n]}y_{n})))
\end{array}  \right.
\end{eqnarray}
 where $T_{[n]}=T_{n(mod~N)}$ and $\{\alpha_{n}\}_{n=1}^{\infty}\subset (0,1)$, $\{\beta_{n}\}_{n=1}^{\infty}\subset [c,d]\subset(0,1)$ satisfying  $\underset{n\to\infty}{\lim}\alpha_n=0$, $\sum^{\infty}_{n=1}\alpha_{n}=\infty$. Then $\{x_{n}\}_{n=1}^{\infty}$ converges strongly to $P^{f}_{\Omega}(x)$, where $P^{f}_{\Omega}$ is the Bregman projection of $C$ onto $\Omega.$
\end{theorem}
\begin{proof}
Let $p=P_{\Omega}^{f}\in\Omega$ from Lemma \ref{le8}, we obtain
$$D_{f}(p,u_{j,n})=D_{f}(p,Res^{f}_{g_j}x_n)\leq D_{f}(p,x_{n})$$
Now from \eqref{rec}, we obtain
\begin{eqnarray}\label{3.3}
\nonumber D_{f}(p,y_{n})&\leq&D_{f}(p,\nabla f^*((1-\alpha_{n})\nabla f(u_{j,n}))
\\&=&\nonumber D_{f}(p,\nabla f^*(\alpha_{n}\nabla f(0)+(1-\alpha_{n})\nabla f(u_{j,n}))
\\&\leq&\nonumber\alpha_{n}D_{f}(p,0)+(1-\alpha_{n})D_{f}(p,u_{j,n})
\\&\leq&\alpha_{n}D_{f}(p,0)+(1-\alpha_{n})D_{f}(p,x_{n})
\end{eqnarray}
Also from \eqref{rec}, \eqref{d} and \eqref{3.3}, we have
\begin{eqnarray}\label{3.3*}
D_{f}(p,x_{n+1})&\leq&\nonumber D_{f}(p,\nabla f^{*}((1-\beta_{n})\nabla f(y_{n})+\beta_{n}\nabla f(T_{[n]}y_{n})))
\\&\leq&\nonumber(1-\beta_{n})D_{f}(p,y_{n})+\beta_{n}D_{f}(p,T_{[n]}y_{n})
\\&\leq&\nonumber(1-\beta_{n})D_{f}(p,y_{n})+\beta_{n}D_{f}(p,y_{n})
\\&=&\nonumber D_{f}(p,y_{n})
\\&\leq&\alpha_{n}D_{f}(p,0)+(1-\alpha_{n})D_{f}(p,x_{n})
\\&\leq&\nonumber\max\{D_{f}(p,0),D_{f}(p,x_{n})\}
\end{eqnarray}
Thus, by induction we obtain
$$D_{f}(p,x_{n+1})\leq\max\{D_{f}(p,0),D_{f}(p,x_{n})\},~~\forall n\geq 0$$
which implies that $\{x_{n}\}$ is bounded and hence $\{y_{n}\},\{T_{[n]}y_{n}\},\{T_{[n]}x_{n}\}$ and $\{u_{j,n}\}$ are all bounded for each $j=1,2,\cdots,m$. Now from \eqref{rec} let $z_{n}:=\nabla f^{*}((1-\alpha_{n})\nabla f(u_{j,n}))$. Furthermore since $\alpha_{n}\to 0$ as $n\to\infty$, we obtain
\begin{eqnarray}\label{3.4}
||\nabla f(z_{n})-\nabla f(u_{j,n})||=\alpha_{n}||(-\nabla f(u_{j,n}))||\to0~~\textrm{as}~~n\to\infty.
\end{eqnarray}
Since $f$ is strongly coercive and uniformly convex on bounded subsets of $E$, $f^*$ is uniformly Fr$\acute{e}$chet differentiable on bounded sets. Moreover, $f^*$ is bounded on bounded sets, from \eqref{3.4}, we obtain
\begin{eqnarray}\label{3.5}
\underset{n\to\infty}{\lim}||z_{n}-u_{j,n}||=0.
\end{eqnarray}
On the other hand, In view of (3) in Theorem \ref{jen1}, we know that $dom f^* = E^*$ and $f^*$ is strongly coercive and
uniformly convex on bounded subsets. Let $s = sup\{||\nabla f(y_n)||,||\nabla f(T_{[n]}y_n)||\}$ and $\rho^{*}_{s}:E^*\to \mathbb{R}$ be the gauge of uniform convexity of the conjugate function $f^*$. Now from \eqref{rec}, Lemma \ref{jen} and \ref{jen3}, we obtain
\begin{eqnarray}\label{3.6}
D_{f}(p,y_{n})&\leq&D_{f}(p,z_{n})=\nonumber V(p,\nabla f(z_{n}))
\\&\leq&\nonumber V(p,\nabla f(z_n)+\alpha_{n}\nabla f(p))+\alpha_{n}\langle -\nabla f(p),z_{n}-p\rangle
\\&=&\nonumber D_{f}(p,\nabla f^{*}((1-\alpha_{n})\nabla f(u_{j,n})+\alpha_{n}\nabla f(p)))\\&&\nonumber+\alpha_{n}\nabla f(p))+\alpha_{n}\langle -\nabla f(p),z_{n}-p\rangle
\\&\leq&\nonumber\alpha_{n} D_{f}(p,p)+(1-\alpha_{n})D_{f}(p,u_{j,n})\\&&\nonumber+\alpha_{n}\langle-\nabla f(p),z_{n}-p\rangle
\\&\leq&(1-\alpha_{n})D_{f}(p,x_{n})+\alpha_{n}\langle -\nabla f(p),z_{n}-p\rangle
\end{eqnarray}
and
\begin{eqnarray}
D_{f}(p,x_{n+1})&\leq&\nonumber D_{f}(p,\nabla f^*((1-\beta_{n})\nabla f(y_{n})+\beta_{n}\nabla f(T_{[n]}y_{n})))
\\&=&\nonumber V(p,(1-\beta_{n})\nabla f(y_{n})+\beta_{n}\nabla f(T_{[n]}y_{n}))
\\&=&\nonumber f(p)-\langle p,(1-\beta_{n})\nabla f(y_{n})+\beta_{n}\nabla f(T_{[n]}y_{n})\rangle\\&&\nonumber
+f^{*}((1-\beta_{n})\nabla f(y_{n})+\beta_{n}\nabla f(T_{[n]}y_{n}))
\\&\leq&\nonumber(1-\beta_{n})f(p)+\beta_{n}f(p)-(1-\beta_{n})\langle p,\nabla f(y_{n})\rangle-\beta_{n}\langle p,\nabla f(T_{[n]}y_{n})
\\&&\nonumber+(1-\beta_{n})f^{*}(\nabla f(T_{[n]}y_{n}))+\beta_{n}f^{*}(\nabla f(T_{[n]}y_{n}))\\&&\nonumber-\beta_{n}(1-\beta_{n})\rho^{*}_{s}(||\nabla f(y_{n})-\nabla f(T_{[n]}y_{n})||)
\\&=&\nonumber(1-\beta_{n})V(p,\nabla f(y_{n}))+\beta_{n}V(p,\nabla f(T_{[n]}y_{n}))\\&&\nonumber-\beta_{n}(1-\beta_{n})\rho^{*}_{s}(||\nabla f(y_{n})-\nabla f(T_{[n]}y_{n})||)
\\&=&\nonumber(1-\beta_{n})D_{f}(p,y_{n})+\beta_{n}D_{f}(p,T_{[n]}y_{n})\\&&\nonumber-\beta_{n}(1-\beta_{n})\rho^{*}_{s}(||\nabla f(y_{n})-\nabla f(T_{[n]}y_{n})||)
\\&\leq&\nonumber(1-\beta_{n})D_{f}(p,y_{n})+\beta_{n}D_{f}(p,y_{n})\\&&\nonumber-\beta_{n}(1-\beta_{n})\rho^{*}_{s}(||\nabla f(y_{n})-\nabla f(T_{[n]}y_{n})||)
\\&=&\nonumber D_{f}(p,y_{n})-\beta_{n}(1-\beta_{n})\rho^{*}_{s}(||\nabla f(y_{n})-\nabla f(T_{[n]}y_{n})||)
\\&\leq&\nonumber(1-\alpha_{n})D_{f}(p,x_{n})+\alpha_{n}\langle -\nabla f(p),z_{n}-p\rangle
\\&&-\beta_{n}(1-\beta_{n})\rho^{*}_{s}(||\nabla f(y_{n})-\nabla f(T_{[n]}y_{n})||)\label{3.7}
\\&\leq&(1-\alpha_{n})D_{f}(p,x_{n})+\alpha_{n}\langle -\nabla f(p),z_{n}-p\rangle\label{3.8}
\end{eqnarray}
Now, we consider two cases:\\
{\bf Case 1.} Suppose that there exists $n_{0}\in \mathbb{N}$ such that $\{D_{f}(p,x_{n})\}$ is non increasing. In this situation $\{D_{f}(p,x_{n})\}$ is convergent. Then from \eqref{3.7} we obtain
\begin{eqnarray}\label{3.9}
\beta_{n}(1-\beta_{n})\rho^{*}_{s}(||\nabla f(y_{n})-\nabla f(T_{[n]}y_{n})||)\to 0~~\textrm{as}~~n\to\infty,
\end{eqnarray}
which implies, by the property of $\rho_{s}$ and since $\beta_{n}\in[c,d]\subset(0,1)$, we obtain
\begin{eqnarray}\label{3.10}
\underset{n\to\infty}{\lim}||\nabla f(y_{n})-\nabla f(T_{[n]}y_{n})||=0
\end{eqnarray}
Since $f$ is strongly coercive and uniformly convex on bounded subsets of $E$, $f^*$ is uniformly Fr$\acute{e}$chet differentiable on bounded sets. Moreover, $f^*$ is bounded on bounded sets, from \eqref{3.10}, we obtain
\begin{eqnarray}\label{3.11}
\underset{n\to\infty}{\lim}||y_{n}-T_{[n]}y_{n}||=0.
\end{eqnarray}
Now from \eqref{eq2}, we obtain
\begin{eqnarray*}
D_{f}(y_{n},T_{[n]}y_{n})&=&D_{f}(p,T_{[n]}y_{n})-D_{f}(p,y_{n})\\&&+\langle \nabla f(T_{[n]}y_{n})-\nabla f(y_{n}),p-y_{n}\rangle
\\&\leq&D_{f}(p,y_{n})-D_{f}(p,y_{n})\\&&+\langle \nabla f(T_{[n]}y_{n})-\nabla f(y_{n}),p-y_{n}\rangle
\end{eqnarray*}
therefore
\begin{eqnarray}\label{3.12}
D_{f}(y_{n},T_{[n]}y_{n})\leq||\nabla f(y_{n})-\nabla f(T_{[n]}y_{n})||||p-y_{n}||\to0~~\textrm{as}~~n\to\infty.
\end{eqnarray}

Also, from \eqref{le8}, we have
\begin{eqnarray}\label{3.13}
D_{f}(x_{n},u_{j,n})&=&\nonumber D_{f}(p,Res^{f}_{g_j}x_n)
\\&\leq&\nonumber D_{f}(p,Res^{f}_{g_j}x_n)-D_{f}(p,x_{n})
\\&\leq&D_{f}(p,x_{n})-D_{f}(p,x_{n})\to0~~\textrm{as}~~n\to\infty.
\end{eqnarray}
Then, we have from Lemma \ref{le3} that
\begin{eqnarray}\label{3.14}
\underset{n\to\infty}{\lim}||x_{n}-u_{j,n}||=0
\end{eqnarray}

Also, from (b) of Lemma \ref{le2}, we have
\begin{eqnarray}\label{3.13*}
D_{f}(y_{n},P_{C}z_{n})&=&\nonumber D_{f}(y_n,z_n)
\\&=&\nonumber D_{f}(y_{n},\nabla f^*(\nabla f(0)+(1-\alpha_{n})\nabla f(u_{j,n}))
\\&\leq&\nonumber\alpha_{n}D_{f}(y_n,0)+(1-\alpha_{n})D_{f}(y_{n},u_{j,n})
\\&\leq&\alpha_{n}D_{f}(y_n,0)+(1-\alpha_{n})D_{f}(u_{j,n},u_{j,n})\to0~~\textrm{as}~~n\to\infty.
\end{eqnarray}
Then, we have from Lemma \ref{le3} that
\begin{eqnarray}\label{3.14*}
\underset{n\to\infty}{\lim}||y_{n}-z_{n}||=0
\end{eqnarray}

From \eqref{3.5} and \eqref{3.14}, we obtain
\begin{eqnarray}\label{3.155}
\underset{n\to\infty}{\lim}||x_{n}-z_{n}||=0.
\end{eqnarray}
From \eqref{3.14*} and \eqref{3.155}, we obtain
\begin{eqnarray}\label{3.15}
\underset{n\to\infty}{\lim}||x_{n}-y_{n}||=0.
\end{eqnarray}

Since $f$ is strongly coercive and uniformly convex on bounded subsets of $E$, $f^*$ is uniformly Fr$\acute{e}$chet differentiable on bounded sets. Moreover, $f^*$ is bounded on bounded sets, from \eqref{3.15}, we obtain
\begin{eqnarray}\label{3.15*}
\underset{n\to\infty}{\lim}||\nabla f(x_{n})-\nabla f(z_{n})||=0.
\end{eqnarray}

Also from \eqref{3.11} and \eqref{3.15}
\begin{eqnarray}\label{3.16}
\underset{n\to\infty}{\lim}||x_{n}-T_{[n]}y_{n}||=0.
\end{eqnarray}

Now from \eqref{eq2} and \eqref{3.3}, we obtain
\begin{eqnarray*}
D_{f}(x_{n},y_{n})&=&D_{f}(p,y_{n})-D_{f}(p,x_{n})+\langle \nabla f(x_{n})-\nabla f(y_{n}),p-x_{n}\rangle
\\&\leq&\alpha_{n}D_{f}(p,0)+(1-\alpha_{n})D_{f}(p,x_{n})-D_{f}(p,x_{n})\\&&+\langle \nabla f(x_{n})-\nabla f(y_{n}),p-x_{n}\rangle
\\&=&\alpha_{n}(D_{f}(p,0)-D_{f}(p,x_{n}))\\&&+\langle \nabla f(x_{n})-\nabla f(y_{n}),p-x_{n}\rangle
\end{eqnarray*}
therefore, from \eqref{3.15*}, we obtain
\begin{eqnarray}\label{3.17}
D_{f}(x_{n},y_{n})&\leq&\nonumber\alpha_{n}(D_{f}(p,0)-D_{f}(p,x_{n}))\\&&+||\nabla f(x_{n})-\nabla f(y_{n})||||p-x_{n}||\to0~~\textrm{as}~~n\to\infty.
\end{eqnarray}
and also
\begin{eqnarray*}
D_{f}(x_{n},T_{[n]}y_{n})&=&D_{f}(p,T_{[n]}y_{n})-D_{f}(p,x_{n})\\&&+\langle \nabla f(x_{n})-\nabla f(T_{[n]}y_{n}),p-x_{n}\rangle
\\&\leq&D_{f}(p,y_{n})-D_{f}(p,x_{n})\\&&+\langle \nabla f(x_{n})-\nabla f(T_{[n]}y_{n}),p-x_{n}\rangle
\\&\leq&\alpha_{n}D_{f}(p,0)+(1-\alpha_{n})-D_{f}(p,x_{n})\\&&+\langle \nabla f(x_{n})-\nabla f(T_{[n]}y_{n}),p-x_{n}\rangle
\\&=&\alpha_{n}(D_{f}(p,0)-D_{f}(p,x_{n}))+\langle \nabla f(x_{n})-\nabla f(T_{[n]}y_{n}),p-x_{n}\rangle
\end{eqnarray*}
thus
\begin{eqnarray}\label{3.18}
D_{f}(x_{n},T_{[n]}y_{n})&\leq&\nonumber\alpha_{n}|D_{f}(p,0)-D_{f}(p,x_{n})|\\&&+ ||\nabla f(T_{[n]}y_{n})-\nabla f(x_{n})||||p-x_{n}||\to0~~\textrm{as}~~n\to\infty.
\end{eqnarray}
Also, from \eqref{3.17}
\begin{eqnarray}\label{3.19}
D_{f}(T_{[n]}x_{n},T_{[n]}y_{n})\leq D_{f}(x_{n},y_{n})\to0~~\textrm{as}~~n\to\infty.
\end{eqnarray}
Then, we have from Lemma \ref{le3} that
\begin{eqnarray}\label{3.20}
\underset{n\to\infty}{\lim}||T_{[n]}x_{n}-T_{[n]}y_{n}||=0.
\end{eqnarray}
Then from \eqref{rec} and \eqref{3.11}, we have
\begin{eqnarray}\label{3.21}
||\nabla f(x_{n+1})-\nabla f(y_{n})||=\beta_{n}||\nabla f(T_{[n]}y_{n})-\nabla f(y_{n})||\to0~~\textrm{as}~~n\to\infty.
\end{eqnarray}

which implies

\begin{eqnarray}\label{3.21*}
||x_{n+1}-y_{n}||\to0~~\textrm{as}~~n\to\infty .
\end{eqnarray}

and

$$||x_{n}-T_{[n]}x_{n}||\leq||x_{n}-y_{n}||+||y_{n}-T_{[n]}y_{n}||+||T_{[n]}y_{n}-T_{[n]}x_{n}||$$
from \eqref{3.11}, \eqref{3.15} and \eqref{3.20}, we obtain
\begin{eqnarray}\label{3.22}
\underset{n\to\infty}{\lim}||x_{n}-T_{[n]}x_{n}||=0.
 \end{eqnarray}
which implies that
\begin{eqnarray}\label{3.22*}
\underset{n\to\infty}{\lim}||\nabla f(x_{n})-\nabla f(T_{[n]}x_{n})||=0.
\end{eqnarray}

Also from \eqref{3.15} and \eqref{3.21*}, we obtain
\begin{eqnarray}\label{3.23}
||x_{n+1}-x_{n}||\leq||x_{n+1}-y_{n}||+||y_{n}-x_{n}||\to0~~\textrm{as}~~n\to\infty.
\end{eqnarray}
But
$$||x_{n+N}-x_{n}||\leq||x_{n+N}-x_{n+N-1}||+||x_{n+N-1}-x_{n+N-2}||+\cdots+||x_{n+1}-x_{n}||\to0$$ as $n\to\infty.$ Hence
\begin{eqnarray}\label{3.23*}
\underset{n\to\infty}{\lim}||x_{n+N}-x_{n}||=0.
\end{eqnarray}

From the uniformly continuous of $\nabla f$, we have from \eqref{3.23} that
\begin{eqnarray}\label{3.24}
\underset{n\to\infty}{\lim}||\nabla f(x_{n+1})-\nabla f(x_{n})||=0.
\end{eqnarray}
From \eqref{eq2}, \eqref{3.3*} and \eqref{3.24}, we obtain
\begin{eqnarray*}
D_{f}(x_{n},x_{n+1})&=&\nonumber D_{f}(p,x_{n+1})-D_{f}(p,x_{n})\\&&\nonumber+\langle \nabla f(x_{n})-\nabla f(x_{n+}),p-x_{n}\rangle
\\&\leq&\nonumber\alpha_{n} D_{f}(p,0)+(1-\alpha_{n})D_{f}(p,x_{n})-D_{f}(p,x_{n})\\&&+\langle \nabla f(x_{n})-\nabla f(x_{n+1}),p-x_{n}\rangle
\end{eqnarray*}
which implies
\begin{eqnarray}\label{3.25}
D_{f}(x_{n},x_{n+1})&\leq&\nonumber\alpha_{n}|D_{f}(p,0)-D_{f}(p,x_{n})|\\&&+||\nabla f(x_{n+1})-\nabla f(x_{n})||||p-x_n||\to0~~\textrm{as}~~n\to\infty.
\end{eqnarray}
Also from quasi-Bregman nonexpansive of $T_{[n]}$, we have
\begin{eqnarray}\label{3.25*}
D_{f}(T_{[n]}x_{n},T_{[n]}x_{n+1})&\leq&D_{f}(x_{n},x_{n+1})\to0~~\textrm{as}~~n\to\infty.
\end{eqnarray}
which implies
\begin{eqnarray}\label{3.25**}
\underset{n\to\infty}{\lim}||T_{[n]}x_{n}- T_{[n]}x_{n+1}||=0.
\end{eqnarray}
and from the uniform continuous of $\nabla f$, we obtain
\begin{eqnarray}\label{3.25***}
\underset{n\to\infty}{\lim}||\nabla f(T_{[n]}x_{n})-\nabla f(T_{[n]}x_{n+1})||=0.
\end{eqnarray}

Also from \eqref{eq2} and \eqref{3.22}, we obtain
\begin{eqnarray}\label{3.26}
D_{f}(x_n,T_{[n]}x_{n})&=&\nonumber D_{f}(p,T_{[n]}x_{n})-D_{f}(p,x_{n})\\&&+\nonumber\langle \nabla f(x_{n})-\nabla f(T_{[n]}x_{n}),p-x_{n}\rangle
\\&\leq&\nonumber D_{f}(p,x_{n})-D_{f}(p,x_{n})\\&&\nonumber+\langle \nabla f(x_{n})-\nabla f(T_{[n]}x_{n}),p-x_{n}\rangle
\\&\leq&||\nabla f(T_{[n]}x_{n})-\nabla f(x_{n})||||p-x_n||\to0~~\textrm{as}~~n\to\infty.
\end{eqnarray}
From \eqref{3.22}, \eqref{3.23} and \eqref{3.25**}, we obtain
\begin{eqnarray}\label{3.26*}
||x_{n}-T_{[n+1]}x_{n}||&\leq&\nonumber||x_{n}-x_{n+1}||+||x_{n+1}-T_{[n+1]}x_{n+1}||\\&&+||T_{[n+1]}x_{n+1}-T_{[n+1]}x_{n}||\to0~~\textrm{as}~~n\to\infty.
\end{eqnarray}
which from uniform continuous of $\nabla f$ implies
\begin{eqnarray}\label{3.27}
\underset{n\to\infty}{\lim}||\nabla f(T_{[n]}x_{n})-\nabla f(T_{[n+1]}x_{n})||=0
\end{eqnarray}
from \eqref{eq2} and \eqref{3.27}, we obtain
\begin{eqnarray}\label{3.28}
D_{f}(x_{n},T_{[n+1]}x_{n})&\leq&\nonumber D_{f}(p,T_{[n+1]}x_{n})-D_{f}(p,x_{n})\\&&\nonumber+\langle \nabla f(x_{n})-\nabla f(T_{[n+1]}x_{n}),p-x_{n}\rangle
\\&\leq&\nonumber D_{f}(p,x_{n})-D_{f}(p,x_{n})\\&&+|| \nabla f(T_{[n+1]}x_{n})-\nabla f(x_{n})||||p-x_{n}||\to0~~\textrm{as}~~n\to\infty.
\end{eqnarray}

From \eqref{eq2}, \eqref{3.25}, \eqref{3.27} and \eqref{3.28}
\begin{eqnarray}\label{3.29}
D_{f}(x_{n+1},T_{[n+1]}x_{n})&=&\nonumber D_{f}(x_{n+1},x_{n})+D_{f}(x_{n},T_{[n+1]}x_{n})\\&&\nonumber+\langle \nabla f(T_{[n+1]}x_{n})-\nabla f(x_{n}),x_{n}-x_{n+1}\rangle
\\&\leq&\nonumber D_{f}(x_{n+1},x_{n})+D_{f}(x_{n},T_{[n+1]}x_{n})\\&&+||\nabla f(T_{[n+1]}x_{n})-\nabla f(x_{n})||||x_{n}-x_{n+1}||\to0~~\textrm{as}~~n\to\infty.
\end{eqnarray}
Also \eqref{eq2}, \eqref{3.25}, and \eqref{3.29}
\begin{eqnarray}\label{3.30}
D_{f}(x_{n},T_{[n+1]}x_{n})&=&\nonumber D_{f}(x_{n},x_{n+1})+D_{f}(x_{n+1},T_{[n+1]}x_{n})\\&&\nonumber+\langle \nabla f(x_{n+1})-\nabla f(T_{[n+1]}x_{n+1}),x_{n+1}-x_{n}\rangle
\\&=&\nonumber D_{f}(x_{n},x_{n+1})+D_{f}(x_{n+1},T_{[n+1]}x_{n})\\&&+||\nabla f(T_{[n+1]}x_{n})-\nabla f(x_{n+1})||||x_{n+1}-x_{n}||\to0~~\textrm{as}~~n\to\infty.
\end{eqnarray}

Using the quasi-Bregman nonexpansivity of $T_{(i)}$ for each $i$, we obtain
, we obtain the following finite table
$$D_{f}(x_{n+N},T_{(n+N)}x_{n+N-1})\to0~~\textrm{as}~~n\to\infty$$
$$D_{f}(T_{(n+N)}x_{n+N-1},T_{(n+N)}T_{(n+N-1)}x_{n+N-2})\to0~~\textrm{as}~~n\to\infty$$
$$\vdots$$
$$D_{f}(T_{(n+N)}\cdots T_{(n+2)}x_{n+1},T_{(n+N)}\cdots T_{(n+1)}x_{n})\to0~~\textrm{as}~~n\to\infty$$
then, applying Lemma \ref{le3} on each line above, we obtain
$$x_{n+N}-T_{(n+N)}x_{n+N-1}\to0~~\textrm{as}~~n\to\infty$$
$$T_{(n+N)}x_{n+N-1}-T_{(n+N)}T_{(n+N-1)}x_{n+N-2}\to0~~\textrm{as}~~n\to\infty$$
$$\vdots$$
$$T_{(n+N)}\cdots T_{(n+2)}x_{n+1}-T_{(n+N)}\cdots T_{(n+1)}x_{n}\to0~~\textrm{as}~~n\to\infty$$

and adding up this table, we obtain
\begin{eqnarray*}
x_{n+N}-T_{(n+N)}T_{(n+N-1)}\cdots T_{(n+1)}x_{n}\to0~\textrm{as}~n\to\infty.
\end{eqnarray*}
Using this and \eqref{3.23*}, we obtain
\begin{eqnarray}\label{luk}
\underset{n\to\infty}{\lim}||x_{n}-T_{(n+N)}T_{(n+N-1)}\cdots T_{(n+1)}x_{n}||=0.
\end{eqnarray}
Also from quasi-Bregman nonexpansive of $T_{(i)}$, for each $i$, we have
\begin{eqnarray}\label{3.32}
{}\\ \nonumber D_{f}(T_{(n+N)}T_{(n+N-1)}\cdots T_{(n+1)}x_{n},T_{(n+N)}T_{(n+N-1)}\cdots T_{(n+1)}y_{n})&\leq&D_{f}(x_{n},y_{n})\to0~~
\end{eqnarray}
$\textrm{as}~~n\to\infty.$ Then, we have from Lemma \ref{le3} that
\begin{eqnarray}\label{3.33}
{}\\ \nonumber T_{(n+N)}T_{(n+N-1)}\cdots T_{(n+1)}x_{n}-T_{(n+N)}T_{(n+N-1)}\cdots T_{(n+1)}y_{n}\to0~~\textrm{as}~~n\to\infty.
\end{eqnarray}
Since
\begin{eqnarray*}
\lefteqn{||y_{n}-T_{(n+N)}T_{(n+N-1)}\cdots T_{(n+1)}y_{n}||\leq||y_{n}-x_{n}||}\\&&+||x_{n}-T_{(n+N)}T_{(n+N-1)}\cdots+ T_{(n+1)}x_{n}||\\&&+||T_{(n+N)}T_{(n+N-1)}\cdots T_{(n+1)}x_{n}-T_{(n+N)}T_{(n+N-1)}\cdots T_{(n+1)}y_{n}||
\end{eqnarray*}
then, from \eqref{3.15}, \eqref{luk} and \eqref{3.33}, we obtain
\begin{eqnarray}\label{3.34}
\underset{n\to\infty}{\lim}||y_{n}-T_{(n+N)}T_{(n+N-1)}\cdots T_{(n+1)}y_{n}||=0.
\end{eqnarray}
Following the argument from \eqref{3.32} to \eqref{3.34} by replacing $y_{n}$ with $z_{n}$ and using \eqref{3.155}, we obtain
\begin{eqnarray}\label{3.344}
\underset{n\to\infty}{\lim}||z_{n}-T_{(n+N)}T_{(n+N-1)}\cdots T_{(n+1)}z_{n}||=0.
\end{eqnarray}

Let $\{x_{n_i}\}$ be a subsequence of $\{x_{n}\}$.
Since $\{x_n\}$ is bounded and $E$ is reflexive, without loss of generality, we may assume that $x_{n_i}\rightharpoonup q$ for some $q\in F$ and since $x_{n}-z_{n}\to0$ as $n\to\infty$, then $z_{n_i}\rightharpoonup q$
Since the pool of mappings of $T_{[n]}$ is finite, passing to a further subsequence if necessary, we may further assume that, for some $i\in\{1,2,\cdots,N\}$, 
 from \eqref{3.344}, we get
$$z_{n_i}-T_{(i+N)}\cdots T_{(i+1)}z_{n_i}\to0~~\textrm{as}~i\to\infty$$
and also
 $$\underset{n\to\infty}{\limsup}\langle -\nabla f(p),z_{n}-p\rangle=\underset{i\to\infty}{\lim}\langle-\nabla f(p),z_{n_i}-p\rangle$$

Noticing that $u_{j,n}=Res^{f}_{g_{j}}(x_{n})$ for each $j=1,2,\cdots,m$, we obtain

$$g_{j}(u_{j,n},y)+\langle y-u_{j,n},\nabla f(u_{j,n})-\nabla f(x_n)\rangle\geq0,~~\forall y\in C$$
Hence
$$g_{j}(u_{j,n_i},y)+\langle y-u_{j,n_i},\nabla f(u_{j,n_i})-\nabla f(x_{n_i})\rangle\geq0,~~\forall y\in C.$$

From the (A2), we note that for each $j=1,2,\cdots,m$,
\begin{eqnarray*}
||y-u_{j,n}||\frac{||\nabla f(u_{j,n_i})-\nabla f(x_{n_i})||}{r_{n_i}}&\geq&
\langle y-u_{j,n},\nabla f(u_{j,n})-\nabla f(x_n)\rangle\\&\geq&-g_{j}(u_{j,n_i})\geq g_{j}(y,u_{j,n_i}),~~\forall y\in C.
\end{eqnarray*}
Taking the limit as $i\to\infty$ in above inequality and from (A4) and $u_{j,n_i}\rightharpoonup q,$ we have $g_{j}(y,q)\leq 0$ for each $j=1,2,\cdots,m$. For $0<t<1$ and $y\in C$, define $y_{t}=ty+(1-t)q$. Noticing that $y,q\in C,$ we obtain $y_t\in C$, which yield that $g_{j}(y_t,q)\leq0$. It follows from (A1) that
$$0=g_{j}(y_t,y_t)\leq t g_{j}(y_t,y)+(1-t)g_{j}(y_t,q)\leq tg_{j}(y_t,y).$$
That is for each $j=1,2,\cdots,m$, we have $g_{j}(y_t,y)\geq0.$\\
Let $t\downarrow0$, from (A3), we obtain $g_{j}(q,y)\geq0$ for any $y\in C,$ for each $j=1,2,\cdots,m$. This implies that $q\in\cap^{m}_{j=1}EP(g_{j}).$
Hence $q\in\Omega.$
 It follows from the definition of the Bregman projection that
\begin{eqnarray*}
\underset{n\to\infty}{\limsup}\langle -\nabla f(p),z_{n}-p\rangle&=&\underset{i\to\infty}{\lim}\langle-\nabla f(p),z_{n_i}-p\rangle
\\&\leq&\langle-\nabla f(p),q-p\rangle\leq0.
\end{eqnarray*}
It follows from Lemma \ref{xu6} and \eqref{3.8} that $D_{f}(p,x_{n})\to 0$ as $n\to\infty$. Consequently, from Lemma \ref{le3}, we obtain $x_{n}\to p$ as $n\to\infty.$\\
{\bf Case 2.} Suppose $D_{f}(p,x_{n})$ is not monotone decreasing sequences, then set $\Phi_{n}:=D_{f}(p,x_{n})$ and let $\tau:\mathbb{N}\to\mathbb{N}$ be a mapping defined for all $n\geq N_{0}$ for some sufficiently large $N_0$ by $$\tau(n):=\max\{k\in\mathbb{N}:k\leq n,\Phi_{k}\leq\Phi_{k+1}\}.$$
Then by Lemma \ref{mainge} $\tau(n)$ is a non-decreasing sequence such that $\tau(n)\to\infty$ as $n\to\infty$ and $\Phi_{\tau(n)}\leq\Phi_{\tau(n)+1}$, for $n\geq N_{0}$. Then from \eqref{3.7} and the fact that $\alpha_{\tau(n)}\to0$, we obtain that
$$\rho^{*}_{s}(||\nabla f(y_{\tau(n)})-\nabla f(T_{[\tau(n)]})y_{\tau(n)}||)\to0~~\textrm{as}~~\tau(n)\to\infty.$$
 Following the same argument as in Case 1, we obtain
$$y_{\tau(n)}-T_{(i+N)}\cdots T_{(i+1)}y_{\tau(n)}\to0~~\textrm{as}~\tau(n)\to\infty$$
and also we obtain
\begin{eqnarray*}
\underset{\tau(n)\to\infty}{\limsup}\langle -\nabla f(p),y_{\tau(n)}-p\rangle\leq0.
\end{eqnarray*}

Then from \eqref{3.8}, we obtain that
\begin{eqnarray}\label{3.41}
0&\leq&D_{f}(p,x_{\tau(n)+1})-D_{f}(p,x_{\tau(n)})\nonumber\\&\leq&\alpha_{\tau(n)}(\langle -\nabla f(p),y_{\tau(n)}-p\rangle-D_{f}(p,x_{\tau(n)}))
\end{eqnarray}

It follows from \eqref{3.41} and $\Phi_{n}\leq\Phi_{\tau(n)+1}$, $\alpha_{\tau(n)}>0$ that
\begin{eqnarray*}
D_{f}(p,x_{\tau(n)})\leq\langle -\nabla f(p),y_{\tau(n)}-p\rangle\to0
\end{eqnarray*}
as $\tau(n)\to\infty$. Thus
$$\underset{\tau(n)\to\infty}{\lim}\Phi_{\tau(n)}=\underset{\tau(n)\to\infty}{\lim}\Phi_{\tau(n)+1}=0.$$
Furthermore, for $n\geq N_{0}$, if $n\neq\tau(n)$ (i.e., $\tau(n)< n$), because $\Phi_{j}>\Phi_{j+1}$ for $\tau(n)+1\leq j\leq n.$ It then follows that for all $n\geq N_{0}$ we have
$$0\leq\Phi_{n}\leq\max\{\Phi_{\tau(n)},\Phi_{\tau(n)+1}\}=\Phi_{\tau(n)+1}.$$
This implies that $\underset{n\to\infty}{\lim}\Phi_{n}=0$, and hence $D_{f}(p,x_{n})\to0$ as $n\to\infty$. Consequently, from Lemma \ref{le3}, we obtain $x_{n}\to p$ as $n\to\infty.$ Therefore from the above two cases, we conclude that $\{x_{n}\}$ converges strongly to $p\in\Omega$ and this complete the proof.

\end{proof}

\end{document}